\documentclass[10pt,twoside,english,reqno,a4paper]{amsart}
\usepackage{listings,graphicx,amsmath,varioref,amscd,amssymb,color,bm,stmaryrd,amsthm,amsfonts,graphics,geometry,latexsym,pgf,pst-all} 
\theoremstyle{plain}

\theoremstyle{plain}
\newtheorem{theorem}{Theorem}[section]

\newtheorem{corollary}[theorem]{Corollary}

\usepackage{geometry}
\usepackage[colorlinks=false]{hyperref}

\theoremstyle{definition}
\newtheorem{defin}[theorem]{Definition}

\newtheorem{remark}[theorem]{Remark}

\theoremstyle{remark}

\usepackage[light,math, condensed]{iwona}
\usepackage[light,math]{iwona}
\usepackage[T1]{fontenc}

\numberwithin{equation}{section}

\def\dis{\displaystyle}
\def\io{\int_{\Omega}} 
\def\supp{\text{\text{supp}}}

\newcommand{\car}[1]{\raise1pt\hbox{$\chi$}_{#1}}

\definecolor{sap}{RGB}{120,36,51}
\def\R{\mathbb{R}}
\def\N{\mathbb{N}}

\def\RN{\mathbb{R}^{N}}
\def\G{\nabla}
\def\div{\text{\text{div}}}

\def\sob{W^{1,p}_{0}(\Omega)}
\def\linf{L^{\infty}(\Omega)}

\def\lp'n{(L^{p'}(\Omega))^{N}}

\def\car#1{\chi_{_{{#1}}}}

\def\norma#1#2{\|#1\|_{\lower 4pt \hbox{$ \scriptstyle #2$ }}}

\author[R. Durastanti]{Riccardo Durastanti}

\author[F. Oliva]{Francescantonio Oliva}

\address[R. Durastanti]{Dipartimento di Scienze di Base e Applicate per l' Ingegneria, ``Sapienza" Universit\`a di Roma, Via Scarpa 16, 00161 Roma, Italy 
\\ riccardo.durastanti@sbai.uniroma1.it}

\address[F. Oliva]{Dipartimento di Scienze di Base e Applicate per l' Ingegneria, ``Sapienza" Universit\`a di Roma, Via Scarpa 16, 00161 Roma, Italy 
\\ francesco.oliva@sbai.uniroma1.it}

\begin{document}

\title{Comparison principle for elliptic equations with mixed singular nonlinearities}

\maketitle

\begin{abstract}

We deal with existence and uniqueness of positive solutions of an elliptic boundary value problem modeled by  
\begin{equation*}
\begin{cases}
\dis -\Delta_p u= \frac{f}{u^\gamma} + g u^q & \mbox{in $\Omega$,} \\
u = 0  & \mbox{on $\partial\Omega$,}
\end{cases}
\end{equation*}
where $\Omega$ is an open bounded subset of $\RN$, $\Delta_p u:=\div(|\G u|^{p-2}\G u)$ is the usual $p$-Laplacian operator, $\gamma\ge 0$ and $0\le q\le p-1$; $f$ and $g$ are nonnegative functions belonging to suitable Lebesgue spaces.

\vskip 0.5\baselineskip
\end{abstract}

\tableofcontents

\section{Introduction}
\label{S1}
  
In this paper we deal with an elliptic problem which simplest model is 
\begin{equation}
\label{pbintro}
\begin{cases}
\dis -\Delta_p u= \frac{f}{u^\gamma} + g u^q & \mbox{in $\Omega$,} \\
u > 0  & \mbox{in $\Omega$,} \\
u = 0  & \mbox{on $\partial\Omega$,}
\end{cases}
\end{equation}
where $\Omega$ is an open bounded subset of $\RN$, $\Delta_p u:=\div(|\G u|^{p-2}\G u)$ is the $p$-Laplacian operator ($1<p<N$), $\gamma,q \ge 0$ are such that $q<p-1$ or $q=p-1$, which correspond to the \textit{sublinear} and to the \textit{linear} behaviour in case $p=2$; here $f,g$ are nonnegative functions belonging to suitable Lebesgue spaces. Clearly the Dirichlet problem \eqref{pbintro} is \textit{singular} since the request that the solution is zero on the boundary of the set implies that the right hand side blows up. For \eqref{pbintro} we are mainly interested to existence and uniqueness of possibly unbounded solutions with \textit{finite energy} (i.e. $u\in W^{1,p}_0(\Omega)$).

\medskip

Let us briefly recall the  mathematical framework concerning problem \eqref{pbintro}; we start with the non-singular case, namely $f\equiv 0$.
\\ The main idea of this paper comes from the seminal paper \cite{bros} where the authors show existence and uniqueness of a solution $u\in H^1_0(\Omega) \cap L^\infty(\Omega)$ to \eqref{pbintro} in case $p=2$, $f\equiv 0$, $q<1$ and $g$ as a bounded nonnegative function. Let us also mention that classical arguments apply once that $u$ is bounded in order to get a $C^1$-solution, at least when the set $\Omega$ is smooth enough.   
Later, in \cite{bo1}, in presence of a possibly  unbounded $g$ and if $q<p-1$, the existence of a solution is proven through an approximation process; here, even in the nonvariational case, it is proven existence of a solution with infinite energy (i.e. $u\not\in W^{1,p}_0(\Omega)$) for rough data $g$. 

\medskip

Let us briefly underline that, when $p=2$, problem \eqref{pbintro} with $f\equiv 0$ is strongly related to the porous media equation in the following way: if $u$ is a solution to \eqref{pbintro} then for some positive constant $c,\tau >0$
$$v(x,t) = c u(x)^q (t+\tau)^{\frac{-q}{1-q}},$$
is a solution to 
$$g(x)u_t - \Delta v^{\frac{1}{q}} = 0.$$

\medskip

On the other side there is a huge literature concerning the purely singular equations, namely $g\equiv 0$. In presence of regular $f$ (say a positive $f\in C^\eta(\Omega)$), \eqref{pbintro} was first treated in these pioneering works \cite{crt, lm, stuart}; here the authors obtain existence and uniqueness of a classical solution (i.e. $u\in C^2(\Omega)\cap C(\overline{\Omega})$). Moreover, among other things, one has that: $u\in C^{2,\eta}(\Omega)$, $u\not\in C^1(\overline{\Omega})$ if $\gamma>1$ and $u\not\in H^1_0(\Omega)$ if $\gamma\ge 3$. Furthermore we refer to \cite{ghl} for more interesting results regarding the regularity of $u$.      
\\ \\ \noindent For what concerns the weak theory of the purely singular case, existence of a distributional solution to \eqref{pbintro} when the $f$ is only a nonnegative function in $L^m(\Omega)$ ($m\ge 1$) is established in \cite{bo}. This solution, if $\gamma\le 1$ (i.e. the\textit{ mild singular} case), attains the boundary datum in the classical sense of Sobolev traces; otherwise, when $\gamma>1$ (i.e. the \textit{strong singular} case), only a power of the solution has zero Sobolev trace and the solution is shown to be locally in the same space. Later, in \cite{do,op2,ddo}, existence of solutions to \eqref{pbintro} is given when the right hand side is of the general form $h(s)f$, with $h$ as a nonnegative and not necessarily monotone function such that $h(s)\le s^{-\gamma}$ near zero and just bounded at infinity. For the nonhomogeneous case in which $q=0$ and $g\not \equiv 0$ we mention \cite{op}.

\medskip

Dealing with uniqueness is more tricky; in \cite{boca} the authors show that the solution is unique in the class of $H^1_0(\Omega)$ and this kind of result has been extended to general nonincreasing nonlinearities and nonlinear operators in \cite{o} for solutions in $W^{1,p}_0(\Omega)$. In \cite{bct}, when $p=2$, the authors show that there is at most one solution to \eqref{pbintro} belonging to $W^{1,1}_0(\Omega)$. 
\\ In \cite{op2}, uniqueness of a distributional solution belonging to $W^{1,1}_{\rm loc}(\Omega)$ (with suitable boundary conditions) is shown for a general measure datum and a nonincreasing nonlinearity. Finally in presence of a very general nonlinear operator and a nonincreasing $h$ it is shown in \cite{ddo} the existence and uniqueness of a renormalized solution for a diffuse measure datum $f$. For further reading on singular problems we refer to \cite{ces,cst,car, diazjfa,gmm, GMM, op}

\medskip

As one should expect the literature concerning \eqref{pbintro} in presence of both $f$ and $g$ not identically zero is less investigated. Already in \cite{stuart} the author proves existence of a classical solution to \eqref{pbintro} when both $f$ and $g$ are regular enough, $p=2$ and $q<1$. In the same direction we refer to \cite{copal} where it is also investigated the superlinear case, which is a completely different framework. The uniqueness of classical solutions to \eqref{pbintro} is shown in \cite{sy} in presence of the Laplacian operator and $q<p-1$; we also refer to \cite{cgr} where, in case of regular $f$ and  $g$ constant, it is proved existence and uniqueness of solutions to \eqref{pbintro} if $q\le 1$; here in the linear case it is proved existence under a smallness assumption on $g$ and nonnexistence otherwise. Then in \cite{locsc}, for $p>1$, through a sub and supersolution argument it is shown existence of solutions to \eqref{pbintro} when the right hand side is of the form $h(u)+k(u)$ and no monotonicity is assumed on $h,k$. In \cite{coco} it is investigated the existence of a solution to \eqref{pbintro} in case $p=2$ when $f$ and $g$ are functions in suitable Lebesgue spaces. 
Let us mention that in \cite{santos}, for $p>1$, the authors show existence and uniqueness of finite energy solutions to \eqref{pbintro} under suitable assumptions on $f,g$. We finally refer to \cite{fs,gg} for more interesting results.

\medskip

The aim of this work is twofold. Firstly, we deal with uniqueness of finite energy solutions by employing the idea contained in \cite{bros}. More precisely we want to prove it for positive solutions to the Dirichlet problem associated to 
\begin{equation}\label{introgenerale}
	-\Delta_p u = F(x,u),
\end{equation}
where $p>1$ and $F$ is a Carath\'eodory function which is possibly unbounded both at the origin and at the infinity and such that
\begin{equation}\label{intromonotonia}
	F(x,s)s^{1-p} \text{ decreases with respect to }s \text{ for a.e. }x\in \Omega. 
\end{equation}
Here the major difficult is dealing with a nonlinear operator when looking for comparison principles. Another issue which needs to be underlined is that the solutions are not required to be bounded; this implying the need of a suitable truncation arguments.      
   It is also worth mentioning that \eqref{intromonotonia} allows to deal with the case $q\le p-1$, at least for positive $f$ if one considers the model case given by \eqref{pbintro}. This result is presented as the comparison principle given by Theorem \ref{comparison} which, as a simple corollary, takes to uniqueness of finite energy solutions. 
   \\
   Other than uniqueness, we are interested to instances of finite energy solutions to \eqref{introgenerale}; this is done both in the mild and in the strongly singular case by means of approximation arguments firstly if $q<p-1$; then we also give an existence result in case $q=p-1$.
   Summarizing, if $q<p-1$, we provide existence of finite energy solutions to equations as in \eqref{pbintro} if $g\in L^{\left(\frac{p^*}{1+q}\right)'}(\Omega)$, $\gamma\le 1$ and $f\in L^{\left(\frac{p^*}{1-\gamma}\right)'}(\Omega),$
     where we mean $L^1(\Omega)$ once that $\gamma=1$.
     \\ Otherwise, we show that if $f\in L^m(\Omega)$ with $1<\gamma<2-\frac{1}{m}$ then the existence is guaranteed under the same assumptions on $g$. Let us also highlight that, as remarked in Section \ref{S4}, there are instances in which one could expect finite energy solutions up to $\gamma< 1+ \frac{p(m-1)}{(p-1)m}$. Finally, once again if $f\in L^{\left(\frac{p^*}{1-\gamma}\right)'}(\Omega)$, we also show the existence of a solution in case $q=p-1$ under a smallness assumption on $g$.

\medskip

Let us mention that formally the change of variable $v= \displaystyle \frac{u^{\gamma+1}}{\gamma+1}$ for $p=2$ takes \eqref{pbintro} to the following equation 
\begin{equation}\label{change}
	-\Delta v  + \frac{\gamma}{\gamma+1} \frac{|\nabla v|^2}{v} = (\gamma+1)^{\frac{\gamma+\theta}{\gamma+1}}gv^{\frac{\gamma+\theta}{\gamma+1}} + f,
\end{equation}
which, for $g=0$, was extensively studied in the past, see for instance \cite{aclmop,am,as,dur,gps}. The previous discussion could be formalized and the existence and uniqueness results given in the current paper could provide information regarding problem \eqref{change}.

\medskip

The plan of the paper is the following: in Section \ref{S2} we state and prove the comparison principle and the associated uniqueness result for problems as in \eqref{pbintro} (Theorem \ref{comparison} and Corollary \ref{uniq}). In Section \ref{S3} we give some existence results; precisely we investigate both the mild and the strongly singular case when $q<p-1$ (Theorem \ref{ex1} and Theorem \ref{ex_gamma>1}); moreover we also treat a case in which $q=p-1$ (Theorem \ref{ex2}).

\subsection{Notation} 

In the entire paper $\Omega$ is an open and bounded subset of $\RN$, with $N\geq 1$. We denote by $\partial A$ the boundary and by $|A|$ the Lebesgue measure of a subset $A$ of $\RN$. By $C^k_c(\Omega)$, with $k\geq 1$, we mean the space of $C^k$ functions with compact support in $\Omega$. \\
For any $q>1$, $q':= \frac{q}{q-1}$ is the H\"older conjugate exponent of $q$, while for any $1\leq p<N$, $p^*=\frac{Np}{N-p}$ is the Sobolev conjugate exponent of $p$. \\
We denote by $\chi_E$ the characteristic function of $E\subset\Omega$, namely 
$$
\chi_E(x)=
\begin{cases} 
1 & x\in E,\\
0 & x\in \Omega\setminus E,
\end{cases}
$$
and by $f^+:=\max(f,0), f^-:= -\min(f,0)$ the positive and the negative part of a function $f$. 
We will widely use the following function defined for a fixed $k>0$ and $s\in\R$
\begin{equation}
\label{defTk}
T_k(s)=\max (-k,\min (s,k)), 
\end{equation}
and
\begin{equation}
\label{Vdelta}
\displaystyle
V_{\delta}(s)=
\begin{cases}
1 \ \ &s\le \delta, \\
\displaystyle\frac{2\delta-s}{\delta} \ \ &\delta <s< 2\delta, \\
0 \ \ &s\ge 2\delta.
\end{cases}
\end{equation}
 If no otherwise specified, we will denote by $C$ several constants whose value may change from line to line. These values will only depend on the data (for instance $C$ may depend on $\Omega$, $N$ and $p$) but they will never depend on the indexes of the sequences we will often introduce. 

\section{Comparison principle and uniqueness}
\label{S2}

Let $1<p<N$ and let us consider the following problem
\begin{equation}
\label{pbmain}
\begin{cases}
\displaystyle -\Delta_p u = F(x,u) & \text{ in }\Omega,\\
u> 0 & \text{ in } \Omega,\\
u=0 & \text{ on } \partial\Omega,
\end{cases}
\end{equation}
where the nonlinearity $F:\Omega\times (0,\infty)\to [0,\infty)$ is a general Carath\'eodory function.

We start specifying the notion of weak solution to \eqref{pbmain}. 
\begin{defin}
\label{defweak}
A positive function $u\in W^{1,p}_0(\Omega)$ is a weak solution to \eqref{pbmain} if $F(x,u)\in L^1_{\rm loc}(\Omega)$ and if

\begin{equation}
\label{def3}
\displaystyle\io |\nabla u|^{p-2}\nabla u\cdot \nabla \varphi=\io F(x,u)\varphi, \ \forall \varphi \in C^1_c(\Omega).
\end{equation}
\end{defin}
In order to deal with uniqueness of solutions, we present a comparison principle for solutions to \eqref{pbmain} provided the right hand side enjoys some monotonicity condition.
In particular let us consider $v_1,v_2$ solutions to 
\begin{equation}
\label{pbmain2}
\begin{cases}
\displaystyle -\Delta_p v_i = G_i(x,v) & \text{ in }\Omega,\\
v_i> 0 & \text{ in } \Omega,\\
v_i=0 & \text{ on } \partial\Omega,
\end{cases}
\end{equation}
where the nonlinearities $G_1,G_2:\Omega\times (0,\infty)\to [0,\infty)$ are Carath\'eodory functions. 
We state the main result of this section. 
\begin{theorem}[Comparison Principle]
\label{comparison}
Let us assume $G_1, G_2$ are nonnegative functions such that either $G_1(x,s)s^{1-p}$ or $G_2(x,s)s^{1-p}$ is decreasing with respect to $s$ and for almost every $x\in \Omega$ and
\begin{equation}
\label{hyp}
G_1(x,s)\leq G_2(x,s)
\end{equation}
for almost every $x\in \Omega$ and for all $s\in (0,\infty)$.
Let $v_1$ and $v_2$ be weak solutions to problem \eqref{pbmain2} with data, respectively, $G_1,G_2$ then $v_1\leq v_2$ almost everywhere in $\Omega$.
\end{theorem}
As a simple corollary of the previous result, one has that uniqueness holds for weak solutions to \eqref{pbmain}.
\begin{corollary}[Uniqueness]
\label{uniq}
Let us assume that $F$ is a nonnegative function such that $F(x,s)s^{1-p}$ is decreasing with respect to $s$ and for almost every $x \in \Omega$. Then there exists at most one weak solution to problem \eqref{pbmain}.
\end{corollary}

\begin{remark}
	Just to give an idea, Corollary \ref{uniq} gives uniqueness of solutions to \eqref{pbmain} when $F$ is modelled  by 
	$$\dis F(x,s)=\frac{f(x)}{s^\gamma}+g(x)s^q, \text{ with } f+g>0 \text{ a.e. in } \Omega,$$
	or by
	$$\dis F(x,s)=\frac{f(x)}{s^\gamma}+g(x) s^{p-1}, \text{ with } f>0 \text{ a.e. in } \Omega,$$
	where $f,g$ are nonnegative functions defined almost everywhere, $\gamma\geq 0$ and $0\leq q<p-1$.
\end{remark}

\subsection{Proof of the comparison principle}
\label{S2bis}

In this section we prove the comparison principle for weak solutions to problem \eqref{pbmain} and, as a consequence, we deduce the uniqueness result, namely Corollary \ref{uniq}. 
\begin{proof}[Proof of Theorem \ref{comparison}]
First of all we need to show that for any weak solution $u$ to \eqref{pbmain}, the formulation \eqref{def3} can be extended for $W^{1,p}$-test functions.
We consider a nonnegative $\varphi\in W^{1,p}_0(\Omega)$ and a sequence of nonnegative functions $\varphi_{\eta,n}\in C^1_c(\Omega)$ such that
\begin{equation*}\label{propapprox}
\begin{cases}
\varphi_{\eta,n} \stackrel{\eta  \to 0}{\to} \varphi_{n} \stackrel{n  \to \infty}{\to} \varphi \ \ \ \text{in } W^{1,p}_0(\Omega)
\\
\supp \varphi_n\subset \subset \Omega: 0\le \varphi_n\le \varphi \ \ \ \text{for all } n\in \mathbb{N}.
\end{cases}
\end{equation*}
An example of such $\varphi_{\eta,n}$ is $\rho_\eta \ast (\varphi \wedge \phi_n)$ ($\varphi \wedge \phi_n:= \inf (\varphi,\phi_n)$) where $\rho_\eta$ is a smooth mollifier and $\phi_n$ is a sequence of nonnegative functions in $C^1_c(\Omega)$ which converges to $\varphi$ in $W^{1,p}_0(\Omega)$.
\\Hence let us take $\varphi_{\eta,n}$ as a test function in \eqref{def3}, yielding to 
\begin{equation*}\label{uni1}
\int_{\Omega} |\nabla u|^{p-2}\nabla u\cdot\nabla 	\varphi_{\eta,n}  = \int_{\Omega}F(x,u) \varphi_{\eta,n}.
\end{equation*} 
We want to pass first $\eta$ to zero and then $n$ to infinity in the previous. 
\\ Since $u\in W^{1,p}_0(\Omega)$ one can pass to the limit the first term recalling that $\varphi_{\eta,n}$ converges to $\varphi_n$ in $W^{1,p}_0(\Omega)$. For the right hand side one has that $F(x,u) \in L^1_{\rm loc}(\Omega)$ that gives that we can pass $\eta\to 0$ since $\varphi_{\eta,n}$ converges $*$-weakly in $L^\infty(\Omega)$ to $\varphi_n$ which has compact support in $\Omega$. Hence we deduce
\begin{equation}\label{uni2}
\int_{\Omega} |\nabla u|^{p-2}\nabla u_\cdot\nabla \varphi_n = \int_{\Omega}F(x,u) \varphi_n.
\end{equation} 	
Now let observe that by the Young inequality
\begin{equation*}\label{uni3}
\int_{\Omega}F(x,u)\varphi_n \le \int_{\Omega} |\nabla u|^{p} + \int_{\Omega}|\nabla \varphi_n|^p,	
\end{equation*}	
and by the Fatou Lemma with respect to $n$, one gets
\begin{equation}\label{l1loc}
\int_{\Omega}F(x,u)\varphi \le C.	
\end{equation}	
Now we take $n\to \infty$ in \eqref{uni2}. For the term on the left hand side we can reason as already done when $\eta \to 0$. For the right hand side of \eqref{uni2} one can easily apply the Lebesgue Theorem since 
$$F(x,u)\varphi_n\le F(x,u)\varphi \overset{\eqref{l1loc}}{\in} L^1(\Omega),$$
which gives
\begin{equation}\label{uni5}
\int_{\Omega} |\nabla u|^{p-2}\nabla u\cdot\nabla \varphi = \int_{\Omega}F(x,u)\varphi,
\end{equation} 		 
for every $\varphi\in W^{1,p}_0(\Omega)$. 

\medskip

Since $v_1$ and $v_2$ are weak solutions to problem \eqref{pbmain2} with data $G_1,G_2$ then, recalling \eqref{uni5}, one  can test both equations with $W^{1,p}_0$-functions. From here we suppose that $G_1(x,s)s^{1-p}$ is decreasing with respect to $s$ for almost every $x \in \Omega$; if one is in the other case, then slight modifications will be needed.\\
Let us fix $\varepsilon >0$ and $k\in\N$ and let us define 
$$
A_{k,\varepsilon}:=\left\{x\in\Omega : 0\leq (v_1(x)+\varepsilon)^p-(v_2(x)+\varepsilon)^p\leq k\right\}, \quad A^c_{k,\varepsilon}=\Omega\setminus A_{k,\varepsilon},
$$
and
$$
A_k=\left\{x\in\Omega : 0\leq v_1(x)^p-v_2(x)^p\leq k\right\}, \quad A^c_k=\Omega\setminus A_k.
$$
We consider the following two functions:
\begin{equation}
\label{phipsi}
\psi_1=\frac{T_k(((v_1+\varepsilon)^p-(v_2+\varepsilon)^p)^+)}{(v_1+\varepsilon)^{p-1}}, \quad 
\psi_2=\frac{T_k(((v_1+\varepsilon)^p-(v_2+\varepsilon)^p)^+)}{(v_2+\varepsilon)^{p-1}},
\end{equation}
where $T_k$ is defined by \eqref{defTk}. Let us also underline that $\psi_1,\psi_2 \in W^{1,p}_0(\Omega)$ (see Remark \ref{insob} below). One has 
\begin{eqnarray*}
\label{gradphi}
\nabla\psi_1 & = & \left(\nabla v_1-p\left(\frac{v_2+\varepsilon}{v_1+\varepsilon}\right)^{p-1}\nabla v_2+(p-1)\left(\frac{v_2+\varepsilon}{v_1+\varepsilon}\right)^p\nabla v_1\right)\chi_{A_{k,\varepsilon}} \nonumber \\
&&-(p-1)\frac{T_k(((v_1+\varepsilon)^p-(v_2+\varepsilon)^p)^+)}{(v_1+\varepsilon)^p}\nabla v_1 \chi_{A^c_{k,\varepsilon}},
\end{eqnarray*}
and
\begin{eqnarray*}
\label{gradpsi}
\nabla\psi_2 & = & -\left(\nabla v_2-p\left(\frac{v_1+\varepsilon}{v_2+\varepsilon}\right)^{p-1}\nabla v_1+(p-1)\left(\frac{v_1+\varepsilon}{v_2+\varepsilon}\right)^p\nabla v_2\right)\chi_{A_{k,\varepsilon}} \nonumber \\
&&-(p-1)\frac{T_k(((v_1+\varepsilon)^p-(v_2+\varepsilon)^p)^+)}{(v_2+\varepsilon)^p}\nabla v_2 \chi_{A^c_{k,\varepsilon}}.
\end{eqnarray*}
We choose $\psi_1$ and $\psi_2$ as test functions in equations solved by, respectively, $v_1$ and $v_2$ and we subtract them yielding to
$$
\begin{aligned}
&\int_{A_{k,\varepsilon}}\left(|\nabla v_1|^p-\left(\frac{v_1+\varepsilon}{v_2+\varepsilon}\right)^p\left|\nabla v_2\right|^p-p\left(\frac{v_1+\varepsilon}{v_2+\varepsilon}\right)^{p-1}|\nabla v_2|^{p-2}\nabla v_2\cdot \left(\nabla v_1-\left(\frac{v_1+\varepsilon}{v_2+\varepsilon}\right)\nabla v_2\right)\right) \nonumber \\
&+\int_{A_{k,\varepsilon}}\left(|\nabla v_2|^p-\left(\frac{v_2+\varepsilon}{v_1+\varepsilon}\right)^p\left|\nabla v_1\right|^p-p\left(\frac{v_2+\varepsilon}{v_1+\varepsilon}\right)^{p-1}|\nabla v_1|^{p-2}\nabla v_1\cdot \left(\nabla v_2-\left(\frac{v_2+\varepsilon}{v_1+\varepsilon}\right)\nabla v_1\right)\right) \nonumber \\
&+(p-1)\int_{A^c_{k,\varepsilon}}\left(\frac{T_k(((v_1+\varepsilon)^p-(v_2+\varepsilon)^p)^+)}{(v_2+\varepsilon)^p}|\nabla v_2|^p-\frac{T_k(((v_1+\varepsilon)^p-(v_2+\varepsilon)^p)^+)}{(v_1+\varepsilon)^p}|\nabla v_1|^p\right) \nonumber \\
&\leq \int_{\Omega}\left(\frac{G_1(x,v_1)}{(v_1+\varepsilon)^{p-1}}-\frac{G_2(x,v_2)}{(v_2+\varepsilon)^{p-1}}\right)T_k(((v_1+\varepsilon)^p-(v_2+\varepsilon)^p)^+).
\end{aligned}
$$
Now using the following classical estimate due to the convexity of the power function (recall that $p>1$)
$$|\xi|^p-|\eta|^p-p|\eta|^{p-2}\eta\cdot (\xi-\eta)\geq 0, \quad \forall \xi,\eta \in \RN,$$
one has 
\begin{equation}\label{dis1}
\begin{aligned}
\lefteqn{(p-1)\int_{A^c_{k,\varepsilon}}\frac{T_k(((v_1+\varepsilon)^p-(v_2+\varepsilon)^p)^+)}{(v_2+\varepsilon)^p}|\nabla v_2|^p}  \\
&\leq  (p-1)\int_{A^c_{k,\varepsilon}}\frac{T_k(((v_1+\varepsilon)^p-(v_2+\varepsilon)^p)^+)}{(v_1+\varepsilon)^p}|\nabla v_1|^p  \\
&+ \int_{\Omega}\left(\frac{G_1(x,v_1)}{(v_1+\varepsilon)^{p-1}}-\frac{G_2(x,v_2)}{(v_2+\varepsilon)^{p-1}}\right)T_k(((v_1+\varepsilon)^p-(v_2+\varepsilon)^p)^+).
\end{aligned}
\end{equation}
Noting that the first term of \eqref{dis1} is nonnegative, we have 
\begin{equation}
\begin{aligned}
\label{dis2}
0 &\leq  (p-1)\int_{A^c_{k,\varepsilon}}\frac{T_k(((v_1+\varepsilon)^p-(v_2+\varepsilon)^p)^+)}{(v_1+\varepsilon)^p}|\nabla v_1|^p  \\
&+ \int_{\Omega}\left(\frac{G_1(x,v_1)}{(v_1+\varepsilon)^{p-1}}-\frac{G_2(x,v_2)}{(v_2+\varepsilon)^{p-1}}\right)T_k(((v_1+\varepsilon)^p-(v_2+\varepsilon)^p)^+)  \\
&\stackrel{\eqref{hyp}}\leq (p-1)\int_{A^c_{k,\varepsilon}}\frac{T_k(((v_1+\varepsilon)^p-(v_2+\varepsilon)^p)^+)}{(v_1+\varepsilon)^p}|\nabla v_1|^p  \\
&+ \int_{\Omega}\left(\frac{G_1(x,v_1)}{(v_1+\varepsilon)^{p-1}}-\frac{G_1(x,v_2)}{(v_2+\varepsilon)^{p-1}}\right)T_k(((v_1+\varepsilon)^p-(v_2+\varepsilon)^p)^+).
\end{aligned}
\end{equation}
Denoting $r_{k,\varepsilon},\tilde{r}_{k,\varepsilon}$ as follows
\begin{eqnarray*}
r_{k,\varepsilon} & = &(p-1)\frac{T_k(((v_1+\varepsilon)^p-(v_2+\varepsilon)^p)^+)}{(v_1+\varepsilon)^p}|\nabla v_1|^p\chi_{A^c_{k,\varepsilon}} \\
&& + \left(\frac{G_1(x,v_1)}{(v_1+\varepsilon)^{p-1}}-\frac{G_1(x,v_2)}{(v_2+\varepsilon)^{p-1}}\right)T_k(((v_1+\varepsilon)^p-(v_2+\varepsilon)^p)^+),
\end{eqnarray*}
and
\begin{eqnarray*}
\tilde{r}_{k,\varepsilon} &=& (p-1)\frac{T_k(((v_1+\varepsilon)^p-(v_2+\varepsilon)^p)^+)}{(v_1+\varepsilon)^p}|\nabla v_1|^p\chi_{A^c_{k,\varepsilon}} \\
&&+\frac{G_1(x,v_1)}{(v_1+\varepsilon)^{p-1}}T_k(((v_1+\varepsilon)^p-(v_2+\varepsilon)^p)^+),
\end{eqnarray*}
then one has
\begin{equation}
\label{dis3}
0\leq r_{k,\varepsilon}^+\leq \tilde{r}_{k,\varepsilon}.
\end{equation}
Since $v_1,v_2$ are positive then one has that $r_{k,\varepsilon}^+$ ($r_{k,\varepsilon}^-$) converges to $r_k^+$ ($r_k^-$ resp.) and $\tilde{r}_{k,\varepsilon}$ converges to $\tilde{r}_k$ almost everywhere in $\Omega$, where
\begin{eqnarray*}
\label{defrk}
r_k &=& (p-1)\frac{T_k((v_1^p-v_2^p)^+)}{v_1^p}|\nabla v_1|^p\chi_{A^c_{k}} + \left(\frac{G_1(x,v_1)}{v_1^{p-1}}-\frac{G_1(x,v_2)}{v_2^{p-1}}\right)T_k((v_1^p-v_2^p)^+),
\end{eqnarray*}
and
\begin{eqnarray*}
\label{defrktil}
\tilde{r}_k &=& (p-1)\frac{T_k((v_1^p-v_2^p)^+)}{v_1^p}|\nabla v_1|^p\chi_{A^c_{k}} + \frac{G_1(x,v_1)}{v_1^{p-1}}T_k((v_1^p-v_2^p)^+).
\end{eqnarray*}
Moreover, using that $T_k(s)\leq s$ for $s\geq 0$, we deduce that
\begin{equation}
\label{dis4}
\frac{T_k(((v_1+\varepsilon)^p-(v_2+\varepsilon)^p)^+)}{(v_1+\varepsilon)^p}\chi_{A^c_{k,\varepsilon}}\leq 1,
\end{equation}
and 
\begin{equation}
\label{dis5}
\frac{T_k(((v_1+\varepsilon)^p-(v_2+\varepsilon)^p)^+)}{(v_1+\varepsilon)^{p-1}} \leq \frac{(v_1+\varepsilon)^p-\varepsilon^p}{(v_1+\varepsilon)^{p-1}} \leq pv_1,
\end{equation}
where the last inequality holds by means of the Langrange Theorem.\\
It follows that
\begin{equation}
\label{dis6}
\tilde{r}_{k,\varepsilon}\stackrel{\eqref{dis4},\eqref{dis5}}\leq (p-1)|\nabla v_1|^p+pG_1(x,v_1)v_1.
\end{equation}
Since $v_1\in W^{1,p}_0(\Omega)$ and from \eqref{l1loc} one has that the right hand side of \eqref{dis6} belongs to $L^1(\Omega)$. This implies, applying the Lebesgue Theorem, that $\tilde{r}_{k,\varepsilon}$ strongly converges to $\tilde{r}_k$ in $L^1(\Omega)$. Now starting from \eqref{dis3} and applying the Vitali Theorem, we obtain that
\begin{equation}
\label{dis7}
r^+_{k,\varepsilon}\to r^+_k \text{ strongly in } L^1(\Omega).
\end{equation}
As regards $r^-_{k,\varepsilon}$, applying the Fatou Lemma, we have
\begin{equation}
\label{dis8}
\limsup_{\varepsilon\to 0} \int_\Omega -r_{k,\varepsilon}^-\leq \int_{\Omega}-r_k^-.
\end{equation}
Hence we deduce that
\begin{equation*}
\displaystyle 0\stackrel{\eqref{dis2}}\leq \limsup_{\varepsilon\to 0} \int_\Omega r_{k,\varepsilon}=\limsup_{\varepsilon\to 0} \int_\Omega (r_{k,\varepsilon}^+-r_{k,\varepsilon}^-)\stackrel{\eqref{dis7},\eqref{dis8}}\leq \int_{\Omega} (r_k^+-r_k^-)=\int_\Omega r_k.
\end{equation*}
Thus, until now, we have shown that
\begin{equation}
\label{dis9}
0\leq \int_\Omega \left((p-1)\frac{T_k((v_1^p-v_2^p)^+)}{v_1^p}|\nabla v_1|^p\chi_{A^c_{k}} + \left(\frac{G_1(x,v_1)}{v_1^{p-1}}-\frac{G_1(x,v_2)}{v_2^{p-1}}\right)T_k((v_1^p-v_2^p)^+)\right).
\end{equation}
Now we pass to the limit in \eqref{dis9} as $k$ tends to infinity. We note that $\chi_{A^c_k}$ tends to $0$ as $k$ tends to infinity. Moreover, using \eqref{dis4} with $\varepsilon=0$, we have
$$
\frac{T_k((v_1^p-v_2^p)^+)}{v_1^p}|\nabla v_1|^p\chi_{A^c_{k}}\leq |\nabla v_1|^p \in L^1(\Omega),
$$
since $v_1\in W^{1,p}_0(\Omega)$.
This implies, applying the Lebesgue Theorem, that 
\begin{equation}
\label{dis10}
\frac{T_k((v_1^p-v_2^p)^+)}{v_1^p}|\nabla v_1|^p\chi_{A^c_{k}} \to 0 \text{ strongly in } L^1(\Omega).
\end{equation}
As regards the second term in the right hand side of \eqref{dis9}, from $G_1(x,s)s^{1-p}$ decreasing with respect to $s$, one has that 
\begin{equation}
\label{dis11}
0\leq -\left(\frac{G_1(x,v_1)}{v_1^{p-1}}-\frac{G_1(x,v_2)}{v_2^{p-1}}\right)T_k((v_1^p-v_2^p)^+), 
\end{equation}
where the right hand side of \eqref{dis11} is increasing in $k$. Applying Beppo Levi's Theorem, we obtain that
\begin{equation}
\label{dis12}
\displaystyle \lim_{k\to \infty}\int_\Omega\left(\frac{G_1(x,v_1)}{v_1^{p-1}}-\frac{G_1(x,v_2)}{v_2^{p-1}}\right)T_k((v_1^p-v_2^p)^+) = \int_\Omega \left(\frac{G_1(x,v_1)}{v_1^{p-1}}-\frac{G_1(x,v_2)}{v_2^{p-1}}\right)(v_1^p-v_2^p)^+.
\end{equation}
By passing to the limit as $k$ tends to infinity in \eqref{dis9}, using \eqref{dis10} and \eqref{dis12}, we have
\begin{equation}
\label{dis13}
0\leq \int_\Omega \left(\frac{G_1(x,v_1)}{v_1^{p-1}}-\frac{G_1(x,v_2)}{v_2^{p-1}}\right)(v_1^p-v_2^p)^+.
\end{equation}
Furthermore from the fact that $G_1(x,s)s^{1-p}$ is decreasing with respect to $s$, one yields to 
\begin{equation*}
\label{dis14}
\left(\frac{G_1(x,v_1)}{v_1^{p-1}}-\frac{G_1(x,v_2)}{v_2^{p-1}}\right)(v_1^p-v_2^p)^+\leq 0\quad \text{a.e. in }\Omega,
\end{equation*}
which, gathered with \eqref{dis13}, gives that $(v_1^p-v_2^p)^+\equiv 0$, that is $v_1\leq v_2$ almost everywhere in $\Omega$.
\end{proof}

\begin{remark}
\label{insob}
Here we show that $\psi_1, \psi_2$ defined by \eqref{phipsi} belong to $\sob$. We focus on $\psi_2$. As a consequence of Lemma $1.1$ contained in \cite{sta} and the fact the $v_1,v_2$ have finite energy, we have that the function $\psi_h$ defined as
$$
\psi_h=\frac{T_k(((v_1+\varepsilon)^p-(T_h(v_2+\varepsilon))^p)^+)}{(v_2+\varepsilon)^{p-1}}
$$
belongs to $\sob$ for every $h\geq 0$. Moreover, by computing its gradient, we get 
\begin{eqnarray*}
\displaystyle \nabla\psi_h & = & -p\nabla v_2\chi_{\{v_2+\varepsilon\leq h\}\cap A_{k,\varepsilon,h}\cap B_h}+p\left(\frac{v_1+\varepsilon}{v_2+\varepsilon}\right)^{p-1}\nabla v_1\chi_{ A_{k,\varepsilon,h}\cap B_h} \\
&&-(p-1)\frac{T_k(((v_1+\varepsilon)^p-(T_h(v_2+\varepsilon))^p)^+)}{(v_2+\varepsilon)^{p}}\nabla v_2,
\end{eqnarray*}
where
$$
A_{k,\varepsilon,h}=\left\{x\in\Omega : 0\leq (v_1(x)+\varepsilon)^p-(T_h(v_2(x)+\varepsilon))^p\leq k\right\}
$$
and
$$
B_h=\left\{x\in\Omega : v_1(x)+\varepsilon\geq T_h(v_2(x)+\varepsilon)\right\}.
$$
It follows from the definition of $A_{k,\varepsilon,h}$ that 
\begin{equation*}
\label{dis15}
\left(\frac{v_1+\varepsilon}{v_2+\varepsilon}\right)^{p-1}\leq \left(\frac{k}{\varepsilon^p}+1\right)^\frac{p-1}{p}.
\end{equation*}
This implies that
$$
|\nabla \psi_h|^p\leq C(p,k,\varepsilon)\left(|\nabla v_2|^p+|\nabla v_1|^p\right),
$$
with $C(p,k,\varepsilon)$ a positive constant dependent only on $p,k,\varepsilon$. Hence, using $v_1,v_2 \in W^{1,p}_0(\Omega)$, we deduce that $\{\psi_h\}$ is bounded in $\sob$ uniformly in $h$. Moreover $\psi_h$ converges to $\psi_2$ almost everywhere in $\Omega$. So that $\psi_h$ converges to $\psi_2$ weakly in $\sob$ and $\psi_2$ belongs to $\sob$. As regards $\varphi$, in a similar way it is possible to prove that $\psi_1$ belongs to $\sob$.
\end{remark}


\section{Existence results in some model equations}
\label{S3}

In this section we give existence results to \eqref{pbmain} for some explicit nonlinearities $F$ of the following form
\begin{equation}
\label{exF1}
F(x,s)=f(x)h(s)+g(x)k(s),
\end{equation}
where $f,g$ are nonnegative functions belonging to suitable Lebesgue space, with $f\not\equiv 0$, and $h,k:(0,\infty)\to[0,\infty)$ are continuous nonnegative functions such that
\begin{equation}
\label{ph}
\exists\,\, \gamma\geq 0, \underline{C}>0: h(s)\leq \frac{\underline{C}}{s^\gamma} \quad \forall s\in(0,\infty), 
\end{equation}
and
\begin{equation}
\label{pk}
\exists\,\, q\ge 0, \overline{C}>0: k(s)\leq\overline{C}s^q \quad \forall s\in(0,\infty).
\end{equation}

\begin{remark}
	\label{kcont}
	Let us observe that \eqref{pk} implies that $k$ can be extended by continuity at $0$ defining $k(0)=0$.
\end{remark}

We underline that we are not assuming any kind of monotonicity on the functions $h,k$ but just some control from the above. Moreover, the case of continuous and bounded $h,k$ are well contained in our existence result. 
\\
For the sake of clarity we reformulate the problem under the assumption \eqref{exF1}:
\begin{equation}
\label{pb1}
\begin{cases}
\displaystyle -\Delta_p u = f(x)h(u)+g(x)k(u) & \text{ in }\Omega,\\
u> 0 & \text{ in } \Omega,\\
u=0 & \text{ on } \partial\Omega.
\end{cases}
\end{equation}

At first we state an existence result in case $\gamma\le 1$ and  $q<p-1$, which we recall that corresponds to the sublinear case when $p=2$ ; let us explicitly note that in the sequel we define $\left(\frac{p^*}{1-\gamma}\right)':=1$ if $\gamma =1$.

In particular one has the following result.
\begin{theorem}
	\label{ex1}
	Let $f\in \displaystyle L^{\left(\frac{p^*}{1-\gamma}\right)'}(\Omega)$ be a nonnegative function not identically zero and let $g\in L^{\left(\frac{p^*}{1+q}\right)'}(\Omega)$ be a nonnegative function. Let $h$ and $k$ be nonnegative continuous functions satisfying \eqref{ph} with $\gamma \le 1$ and \eqref{pk} with $q<p-1$ respectively. Then there exists at least one weak solution to problem \eqref{pb1}.
\end{theorem}

\begin{remark}
	\label{1f0}
	In the case $f\equiv 0$, if $k$ is an increasing function satisfying \eqref{pk}, the existence of a weak solution to \eqref{pb1} is contained in \cite{bo1}.
\end{remark}

Next we deal with the more difficult case of a strong singularity; here, in order to deduce an existence result, we need some regularity on the $\Omega$.

\begin{theorem}\label{ex_gamma>1}
	Let $\Omega$ satisfy the interior ball condition and let $f\in L^m(\Omega)$ with $m> 1$ be a nonnegative  function and let $g\in L^{\left(\frac{p^*}{1+q}\right)'}(\Omega)$ be a nonnegative function.   Let $h$ and $k$ be nonnegative continuous functions satisfying \eqref{ph} with $1<\gamma<2 - \frac{1}{m}$ and \eqref{pk} with $q<p-1$ respectively. Then there exists at least one weak solution to problem \eqref{pb1}.
\end{theorem}

Finally we also dealt with $q=p-1$. In the next result we denote by $C_p$ the best constant for the Poincar\'e inequality in $\Omega$; we also recall that $\overline{C}$ is the one defined by \eqref{pk}.

\begin{theorem}
	\label{ex2}
	Let $f\in\displaystyle L^{\left(\frac{p^*}{1-\gamma}\right)'}(\Omega)$ be a nonnegative function not identically zero and let $g$ such that $||g||_{L^\infty(\Omega)}< (\overline{C}C^p_p)^{-1}$. Let $h$ and $k$ be nonnegative continuous functions satisfying \eqref{ph} with $\gamma \le 1$ and \eqref{pk} with $q=p-1$ respectively. Then there exists at least one weak solution to problem \eqref{pb1}.
\end{theorem}

\begin{remark}\label{rem_exun}
	Collecting the existence results contained in Theorems \ref{ex1}, \ref{ex_gamma>1} and \ref{ex2} with the uniqueness result contained in Corollary \ref{uniq} we obtain that there exists a unique solution $u \in W^{1,p}_0(\Omega)$ to
	$$-\Delta_p u = F(x,u),$$ 
	under the assumptions of Theorems \ref{ex1} and \ref{ex_gamma>1} in case $(h(s)+k(s)){s^{1-p}}$  is decreasing with respect to $s$ and requiring that $f+g$ is almost everywhere positive in $\Omega$.
	\\ Moreover under the assumptions of Theorem \ref{ex2} one has a unique solution if  $h(s){s^{1-p}}$ is decreasing with respect to $s$ and $f$ is almost everywhere positive in $\Omega$.
\end{remark}

\subsection{Proof of the existence results}

 Let us introduce the following scheme of approximation
\begin{equation}
\label{pb1n}
\begin{cases}
\displaystyle -\Delta_p u_n = f_n h_n(u_n) + g_n k_n(u_n) & \text{ in }\Omega,\\
u_n=0 & \text{ on }\partial\Omega,
\end{cases}
\end{equation}
where $f_n=T_n(f)$ and $g_n=T_n(g)$. Moreover, defining $h(0):= \lim_{s\to 0}h(s)$, we set
$$
h_n(s)=\begin{cases}
T_n(h(s)) & \text{ for }s>0, \\
\min(n,h(0)) & \text{ otherwise},
\end{cases}\quad
\text{and}\quad
k_n(s)=\begin{cases}
T_n(k(s)) & \text{ for }s>0, \\
0 & \text{ otherwise}.
\end{cases}
$$
The existence of a weak solution $u_{n} \in \sob$ is guaranteed by \cite{ll}. Moreover, by Theorem 4.2 of \cite{sta}, we get that $u_n$ is bounded and, since the right hand side of \eqref{pb1n} is nonnegative, that $u_{n}$ is nonnegative.

\begin{remark}
	\label{increas}
	Under the assumptions of Remark \ref{rem_exun} one has that the approximating sequence $\{u_n\}$ is increasing w.r.t. $n$. Indeed defining $F_n(x,s)=f_n(x)h_n(s)+g_n(x)k_n(s)$ one deduces that for every $n$ in $\N$
	$$
	F_n(x,s)\leq F_{n+1}(x,s)\quad \forall s\in (0,\infty) \text{ and for a.e. }x\in\Omega.
	$$
	This allows to apply Theorem \ref{comparison}, yielding to 
	$$
	u_n\leq u_{n+1} \qquad \forall n\in\N.
	$$
\end{remark}

\medskip

\begin{proof}[Proof of Theorem \ref{ex1}]
We divide the proof in two steps. In the first one, we show a priori estimates on $u_n$, solutions to \eqref{pb1n}. In the second one we pass to the limit our approximation in order to deduce the existence of a weak solution to \eqref{pb1}. \\
\textbf{Step 1.}
Let us choose $u_n$ as a test function in the weak formulation of \eqref{pb1n} and from the H\"older inequality and from \eqref{ph},\eqref{pk}, one gets
\begin{eqnarray}
\label{app1}
\int_{\Omega} |\nabla u_n|^p &=& \int_{\Omega} \left(f_n h_n(u_n)u_n + g_n k_n(u_n)u_n\right) \nonumber \\
&\leq & \underline{C}\int_{\Omega} f_n u_n^{1-\gamma} + \overline{C}\int_{\Omega} g_n u_n^{1+q} \nonumber \\
&\leq & \underline{C}||f||_{L^{\left(\frac{p^*}{1-\gamma}\right)'}(\Omega)} ||u_n||^{1-\gamma}_{L^{p^*}(\Omega)}+ \overline{C}||g||_{L^{\left(\frac{p^*}{q+1}\right)'}(\Omega)} ||u_n||^{q+1}_{L^{p^*}(\Omega)}. 
\end{eqnarray} 
If $\displaystyle ||u_n||_{L^{p^*}(\Omega)}\leq 1$, we deduce that $\{u_n\}$ is bounded in $\sob$ uniformly in $n$. Otherwise, recalling that $0\leq 1-\gamma<q+1<p$, we obtain, applying the Sobolev embedding Theorem on the left-hand side of \eqref{app1}, that
\begin{equation}
\label{app2}
\displaystyle ||u_n||_{L^{p^*}(\Omega)}^p\leq C\left(||f||_{L^{\left(\frac{p^*}{1-\gamma}\right)'}(\Omega)}+||g||_{L^{\left(\frac{p^*}{q+1}\right)'}(\Omega)}\right)||u_n||^{q+1}_{L^{p^*}(\Omega)}.
\end{equation}
This implies, dividing by $\displaystyle ||u_n||^{q+1}_{L^{p^*}(\Omega)}$ both members of \eqref{app2}, that $\{u_n\}$ is bounded in $\displaystyle L^{p^*}(\Omega)$ uniformly in $n$. It follows from \eqref{app1} that $\{u_n\}$ is bounded in $\sob$ with respect to $n$. This implies that there exists a nonnegative function $u$ in $\sob$ such that $u_n \to u$  weakly in $\sob$ and almost everywhere in $\Omega$.
Let us take $0\le \varphi\in \sob$ as test function in the weak formulation of \eqref{pb1n}; one obtains, using the Young inequality, that 
\begin{equation}\label{stimaL1}
	\int_\Omega \left(f_nh_n(u_n)+g_nk_n(u_n)\right)\varphi =\int_\Omega |\nabla u_n|^{p-2}\nabla u_n\cdot\nabla\varphi\leq \frac{1}{p'}\int_{\Omega} |\nabla u_n|^p+\frac{1}{p}\int_\Omega |\nabla\varphi|^p \leq C.
\end{equation}
Hence $\{f_nh_n(u_n)+g_nk_n(u_n)\}$ is bounded in $L^1_{\rm loc}(\Omega)$ and, applying Theorem 2.1 of \cite{bm}, that $\nabla u_n$ converges almost everywhere in $\Omega$ to $\nabla u$. 

\medskip

\textbf{Step 2.} In this second step we prove that $u$ obtained in the first step is a weak solution to \eqref{pb1}. \\

First of all we apply the Fatou Lemma in \eqref{stimaL1} in order to deduce that
\begin{equation*}
\label{app6}
\int_{\Omega} (fh(u)+gk(u))\varphi \leq \liminf_{n\to\infty}\int_{\Omega}\left(f_nh_n(u_n)+g_nk_n(u_n)\right)\varphi \le C,
\end{equation*}
hence $(fh(u)+gk(u))\varphi\in L^1(\Omega)$ for any nonnegative $\varphi \in \sob$. As a consequence, if $h(s)$ is unbounded as $s$ tends to $0$, we deduce that 
\begin{equation}
\label{app7}
\{u=0\}\subset\{f=0\},
\end{equation}
up to a set of zero Lebesgue measure. \\
From now on, we assume that $h(s)$ is unbounded as $s$ tends to $0$. Let $\varphi$ be a nonnegative function in $\sob\cap\linf$. Choosing it as test function in the weak formulation of \eqref{pb1n} we have
\begin{equation}
\label{app4}
\int_\Omega |\nabla u_n|^{p-2}\nabla u_n\cdot\nabla\varphi=\int_\Omega (f_nh_n(u_n)+g_nk_n(u_n))\varphi.
\end{equation}
We want to pass to the limit in \eqref{app4} as $n$ tends to infinity. We fix $\delta>0$ and we decompose the right hand side in the following way:
\begin{eqnarray}
\label{app11}
\int_\Omega (f_nh_n(u_n)+g_nk_n(u_n))\varphi &=&\int_{\{u_n\leq \delta\}}(f_nh_n(u_n)+g_nk_n(u_n))\varphi \nonumber \\
&&+\int_{\{u_n>\delta\}}(f_nh_n(u_n)+g_nk_n(u_n))\varphi.
\end{eqnarray}
Therefore we have, thanks to Lemma $1.1$ contained in \cite{sta}, that $V_\delta(u_n)\varphi$ belongs to $\sob$, where $V_\delta(s)$ is defined by \eqref{Vdelta}. So we take it as test function in the weak formulation of \eqref{pb1n} and we obtain
\begin{equation}
\begin{aligned}
\label{app8}
\int_{\{u_n\leq\delta\}}(f_nh_n(u_n)+g_nk_n(u_n))\varphi &\stackrel{\eqref{Vdelta}}\leq \int_\Omega(f_nh_n(u_n)+g_nk_n(u_n))V_\delta(u_n)\varphi \nonumber \\
&=\int_{\Omega} |\nabla u_n|^{p-2}\nabla u_n\cdot\nabla\varphi V_\delta(u_n)-\frac{1}{\delta}\int_{\{\delta<u_n<2\delta\}} |\nabla u_n|^p\varphi \nonumber \\
&\leq\int_{\Omega} |\nabla u_n|^{p-2}\nabla u_n\cdot\nabla\varphi V_\delta(u_n)
\end{aligned}
\end{equation}
Using that $V_\delta$ is bounded we deduce that $|\nabla u_n|^{p-2}\nabla u_n V_\delta(u_n)$ converges to $|\nabla u|^{p-2}\nabla u V_\delta(u)$ weakly in $L^{p'}(\Omega)^N$ as $n$ tends to infinity. This implies that
\begin{equation}
\label{app9}
\lim_{n\to\infty}\int_{\{u_n\leq\delta\}}(f_nh_n(u_n)+g_nk_n(u_n))\varphi\leq \int_{\Omega}|\nabla u|^{p-2}\nabla u\cdot\nabla\varphi V_\delta(u).
\end{equation}
Since $V_\delta(u)$ converges to $\chi_{\{u=0\}}$ a.e. in $\Omega$ as $\delta$ tends to $0$ and since $u\in\sob$, then $|\nabla u|^{p-2}\nabla u\cdot\nabla\varphi V_\delta(u)$ converges to $0$ a.e. in $\Omega$ as $\delta$ tends to $0$. Applying the Lebesgue Theorem on the right hand side of \eqref{app9} we obtain that
\begin{equation}
\label{app10}
\lim_{\delta\to 0^+}\lim_{n\to\infty}\int_{\{u_n\leq\delta\}}(f_nh_n(u_n)+g_nk_n(u_n))\varphi=0.
\end{equation}
As regards the second term in the right hand side of \eqref{app11} we have 
\begin{equation}
\label{app12}
0\leq (f_nh_n(u_n)+g_nk_n(u_n))\chi_{\{u_n>\delta\}}\varphi \stackrel{\eqref{ph},\eqref{pk}}\leq \left(f\sup_{\{s>\delta\}}h(s)+\overline{C}gu_n^q\right)\varphi.
\end{equation}
Thanks to the a priori estimates on $u_n$ and using the Rellich-Kondrakov Theorem, we deduce, up to subsequence, that $u_n^q$ converges to $u^q$ strongly in $L^{\left(\frac{p^*}{1+q}\right)}(\Omega)$. Since $g$ belongs to $\displaystyle L^{\left(\frac{p^*}{1+q}\right)'}(\Omega)$ this implies that the right hand side of \eqref{app12} converges strongly in $L^1(\Omega)$. Moreover we can always assume that $\delta\not\in\{\alpha: |\{u=\alpha\}|>0\}$ which is at most a countable set. As a consequence $\chi_{\{u_n>\delta\}}$ converges to $\chi_{\{u>\delta\}}$ a.e. in $\Omega$. Hence, using once again the Lebesgue Theorem in \eqref{app12}, we deduce first that $(f_nh_n(u_n)+g_nk_n(u_n))\chi_{\{u_n>\delta\}}\varphi$ converges to $(fh(u)+gk(u))\chi_{\{u>\delta\}}\varphi$ strongly in $L^1(\Omega)$ as $n$ tends to infinity, then, since $(fh(u)+gk(u))\varphi$ belongs to $L^1(\Omega)$, that $(fh(u)+gk(u))\chi_{\{u>\delta\}}\varphi$ converges to $(fh(u)+gk(u))\chi_{\{u>0\}}\varphi$ strongly in $L^1(\Omega)$ as $\delta$ tends to $0$. Recalling \eqref{app7} and Remark \ref{kcont}, we conclude that
\begin{eqnarray}
\label{app13}
\lim_{\delta\to 0^+}\lim_{n\to\infty}\int_{\{u_n>\delta\}}(f_nh_n(u_n)+g_nk_n(u_n))\varphi&=&\int_{\{u>0\}}(fh(u)+gk(u))\varphi \nonumber \\
&\stackrel{\eqref{app7}}=&\int_{\Omega}(fh(u)+gk(u))\varphi.
\end{eqnarray}
Finally, using the weak convergence of $u_n$ in $W^{1,p}_0(\Omega)$ and the almost everywhere convergence of the gradients one can pass to the limit as $n\to \infty$ in the left hand side of \eqref{app4}. Moreover, by \eqref{app10} and by \eqref{app13}, we can also take to the limit the right hand side of \eqref{app4} in order to deduce that
\begin{equation}
\label{app14}
\int_{\Omega}|\nabla u|^{p-2}\nabla u\cdot\nabla\varphi=\int_{\Omega}(fh(u)+gk(u))\varphi \qquad \forall\,\, 0\leq\varphi\in\sob\cap\linf.
\end{equation}
Moreover, decomposing any $\varphi=\varphi^+-\varphi^-$, and using that \eqref{app14} is linear in $\varphi$, we deduce that \eqref{app14} holds for every $\varphi\in \sob \cap L^\infty(\Omega)$. \\
We treated $h(s)$ unbounded as $s$ tends to $0$, as regards bounded function $h$ the proof is easier and the only difference deals with the passage to the limit in the right hand side of \eqref{app4}. We can avoid introducing $\delta$ and we can substitute \eqref{app12} with 
\begin{equation*}
\label{app16}
0\leq (f_nh_n(u_n)+g_nk_n(u_n))\varphi\leq \left(f||h||_{\linf}+\overline{C}gu_n^q\right)\varphi.
\end{equation*}
Using the same argument above we have that $(f_nh_n(u_n)+g_nk_n(u_n))\varphi$ converges to $(fh(u)+gk(u))\varphi$ strongly in $L^1(\Omega)$ as $n$ tends to infinity. Then we can conclude as in case of an unbounded $h$.  \\
Finally, it follows from \eqref{app14} and using the strong maximum principle that $u>0$ almost everywhere in $\Omega$. This implies that $u$ is a weak solution to \eqref{pb1n}.
\end{proof}

Now we prove Theorem \ref{ex_gamma>1}, namely the case where $\gamma>1$; here we need a more refined argument in order to control the possibly singular term.

\begin{proof}[Proof of Theorem \ref{ex_gamma>1}]
	We take $u_n$ as a test function in \eqref{pb1n} yielding to 
	\begin{equation}\label{313}
	\begin{aligned}
	\displaystyle \int_{\Omega}|\nabla u_n|^p &\le \int_{\Omega} \left(f_n h_n(u_n)u_n + g_n k_n(u_n)u_n\right) \le \underline{C}\int_{\Omega} f_n u_n^{1-\gamma} + \overline{C}\int_{\Omega} g_n u_n^{1+q}
	\\
	&\le \underline{C}\int_{\Omega} f_n u_n^{1-\gamma} +  \overline{C}||g||_{L^{\left(\frac{p^*}{q+1}\right)'}(\Omega)} ||u_n||^{q+1}_{L^{p^*}(\Omega)}.
	\end{aligned}
	\end{equation}
	Hence, we just need an estimate on the first term of the right hand side of \eqref{313}. First of all let us observe that there exists a nonincreasing and continuous function $\underline{h}:[0,\infty)\rightarrow [0,\infty)$ such that
	\begin{equation*}
	\underline{h}(s)\le h_n(s)\,,\ \ \forall \ s>0, \ \ n\in\mathbb{N}\,.
	\label{funzionesotto}
	\end{equation*}
	For the construction of such $\underline{h}$ we refer to \cite{do}. Hence let us consider  $v_n\in W^{1,p}_0(\Omega)\cap L^\infty(\Omega)$ solution to  
	$$
	\begin{cases}
	\displaystyle -\Delta_p v_n = \underline{h}(v_n)f_n &  \text{in}\, \Omega, \\
	v_{n}=0 & \text{on}\ \partial \Omega.
	\label{pbv}
	\end{cases}
	$$
 	Once again, reasoning as in \cite{do,ddo}, one has that $v_{n}$ is nondecreasing  with respect to $n$ and also that $u_n\ge v_n\ge v_{1}$. Moreover, it follows from the Hopf Lemma (see Lemma A$.3$ of \cite{sak}) that
	\begin{equation*}
	v_{1}(x)\ge C\delta(x), \text{  for  } x \in \Omega,
	\label{stimau1}
	\end{equation*}
	where $\delta(x)$ is the distance function from the boundary $\partial \Omega$.\\
	Thanks to the previous we can finally estimate the term on the right hand side of \eqref{313} as follows:
	$$\int_{\Omega}f_n u_n^{1-\gamma} \le C^{1-\gamma}||f||_{L^m(\Omega)}\left( \int_{\Omega} \frac{1}{\delta^{(\gamma-1)m'}}\right)^{\frac{1}{m'}},$$
	which is finite since $\displaystyle \gamma<2 - \frac{1}{m}.$   
	This allows to have an estimate on $u_n$ in $W^{1,p}_0(\Omega)$ which is independent on $n$. Hence one can reason as in Step $2$ of Theorem \ref{ex1} in order to deduce the existence of a weak solution.
\end{proof}

Finally we prove Theorem \ref{ex2}.
\begin{proof}[Proof of Theorem \ref{ex2}]
We choose $u_n$ itself as a test function in the weak formulation of \eqref{pb1n} and applying the H\"older inequality and the Poincar\'e inequality, we get 
\begin{eqnarray*}
\label{app20}
\int_{\Omega}|\nabla u_n|^p &\stackrel{\eqref{ph}}\leq& \underline{C}\int_{\Omega}fu_n^{1-\gamma}+ \overline{C}||g||_{L^\infty(\Omega)} \int_{\Omega}u_n^p \nonumber \\
&\leq& \underline{C}||f||_{L^{\left(\frac{p^*}{1-\gamma}\right)'}(\Omega)} ||u_n||^{1-\gamma}_{L^{p^*}(\Omega)}+\overline{C} ||g||_{L^\infty(\Omega)}C^p_p\int_\Omega |\nabla u_n|^p,
\end{eqnarray*}
which, recalling $1-\overline{C} ||g||_{L^\infty(\Omega)}C^p_p>0$, implies that 
\begin{equation}
\label{app21}
\int_{\Omega}|\nabla u_n|^p\leq \frac{\underline{C}}{1-\overline{C} ||g||_{L^\infty(\Omega)}C^p_p}
	||f||_{L^{\left(\frac{p^*}{1-\gamma}\right)'}(\Omega)} ||u_n||^{1-\gamma}_{L^{p^*}(\Omega)}.
\end{equation}
Applying the Sobolev embedding Theorem in the right hand side of \eqref{app21}, we have
\begin{equation*}
\label{app22}
||u_n||_{\sob}^p\leq  \frac{\underline{C}\mathcal{S}^{1-\gamma}}{1-\overline{C} ||g||_{L^\infty(\Omega)}C^p_p} ||f||_{L^{\left(\frac{p^*}{1-\gamma}\right)'}(\Omega)} ||u_n||^{1-\gamma}_{\sob},
\end{equation*}
where $\mathcal{S}$ is the constant of the embedding. Since $p>1-\gamma$ it follows that $\{u_n\}$ is bounded in $\sob$. So, up to subsequence, we have
$u_n \to u$  weakly in $\sob$ and almost everywhere in $\Omega$.
Finally we can repeat the argument of Step $2$ of Theorem \ref{ex_gamma>1} in order to conclude that $u$ is a solution to \eqref{pb1}.
\end{proof}

\subsection{A concluding remark}
\label{S4}

Here we underline that the result in Theorem \ref{ex_gamma>1} is not sharp, at least in the model case. 
Let $\Omega\subset \mathbb{R}^N$ be open and bounded with smooth boundary and let us consider the following problem
\begin{equation}
\label{pbfinal}
\begin{cases}
\dis -\Delta u= \frac{f}{u^\gamma} + g u^q & \mbox{in $\Omega$,} \\
u = 0  & \mbox{on $\partial\Omega$,}
\end{cases}
\end{equation}
where $\gamma>1$, $q<1$, $0<f\in L^1(\Omega)$ and $g\in L^\infty(\Omega)$ nonnegative.
We recall the following result proven in \cite{s}.
\begin{theorem}\label{teocin}
	Let $\gamma >1$, $q<1$ and let us suppose that there exists a function $u_0\in H^1_0(\Omega)$ such that 
	\begin{equation}\label{cin}
	\int_{\Omega} f u_0^{1-\gamma}<\infty.
	\end{equation}
	Then there exists a solution $u\in H^1_0(\Omega)$ to \eqref{pbfinal}.
\end{theorem}

Using the previous result we have the following existence theorem:
\begin{theorem}\label{cinmio}
		Let $f\in L^m(\Omega)$ with $m> 1$ be a nonnegative  function and let $g\in L^{\infty}(\Omega)$ be a nonnegative function.   Let $1<\gamma<3 - \frac{2}{m}$ and  $q<1$ then there exists a solution to problem \eqref{pbfinal}.
\end{theorem}
\begin{proof}
	In order to show the existence of a solution we employ \eqref{cin} with $u_{0}=\delta(x)^{t}$for some $t>\frac{1}{2}$ and where $\delta(x)$ is the distance function from the boundary $\partial \Omega$. Indeed, one can show 
	that an application of  the H\"older inequality		
	$$
	\int_\Omega f u_0^{1-\gamma}\le  C\io\delta^{t(1-\gamma)m'}
	$$ 
	and the last integral is finite thanks on the assumption $\gamma< 3 - \frac{2}{m}$. 
\end{proof}

We also remark that, in \cite{ch}, Theorem \ref{teocin} is extended for the case of the $p$-Laplacian operator with $p>2$. In this case one can show that a similar result to Theorem \ref{cinmio} with $1<\gamma<1+ \frac{p(m-1)}{(p-1)m}$.


\begin{thebibliography}{99}
\bibitem{aclmop} D. Arcoya, J. Carmona, T. Leonori, P. J. Mart\'inez-Aparicio, L. Orsina and F. Petitta, Existence and nonexistence of solutions for singular quadratic quasilinear equations, J. Differ. Equ. 246 (2009) 4006–4042.
\bibitem{am} D. Arcoya and P. J. Mart\'inez-Aparicio, Quasilinear equations with natural growth, Rev. Mat. Iberoam. 24 (2008)
597–616.
\bibitem{as} D. Arcoya and S. Segura de Le\'on, Uniqueness of solutions for some elliptic equations with a quadratic
gradient term, ESAIM Control Optim. Calc. Var. 16 (2010) 327–336.
\bibitem{bm} L. Boccardo and F. Murat, Almost everywhere convergence of the gradients of solutions to elliptic and parabolic equations, Nonlinear Anal. 19 (1992) 581-597.

\bibitem{boca} L. Boccardo and J. Casado-Diaz, Some properties of solutions of some semilinear elliptic singular problems and applications to the $G$-convergence, Asymptot. Anal. 86 (2014) 1-15.

\bibitem{bo1} L. Boccardo and L. Orsina, Sublinear elliptic equations in $L^s$, Houston Math. J. 20 (1994) 99-114.

\bibitem{bo}L. Boccardo and L. Orsina, Semilinear elliptic equations with singular nonlinearities, Calc. Var. and PDEs  37 (3-4) (2010) 363-380.
	
\bibitem{bct} B. Brandolini, F. Chiacchio and C. Trombetti, Symmetrization for singular semilinear elliptic equations, Ann. Mat. Pura Appl. (4) 193 (2) (2014) 389-404.
	
\bibitem{bros} H. Brezis and L. Oswald, Remarks on sublinear elliptic equations, Nonlinear Anal. 10 (1986) 55-64.
\bibitem{ces} A. Canino, F. Esposito and B. Sciunzi, On the Höpf boundary lemma for singular semilinear elliptic equations, J. Diff. Equ. 266 (9) (2019) 5488-5499. 
\bibitem{cst} A. Canino, B. Sciunzi and  A. Trombetta, Existence and uniqueness for $p$-Laplace equations involving singular nonlinearities, NoDEA Nonlinear Differential Equations Appl. (2016) 23:8.
\bibitem{car} J. Carmona and P.J. Mart\'inez-Aparicio, A singular semilinear elliptic equation with a variable exponent, Advanced Nonlinear Studies 16 (2016) 491-498.
\bibitem{cgr} F. C$\hat{\text{i}}$rstea, M. Ghergu and V. R$\breve{\text{a}}$dulescu, Combined effects of asymptotically linear and singular nonlinearities in bifurcation problems of Lane-Emden-Fowler type, J. Math. Pures Appl. (9) 84 (4) (2005) 493-508.
\bibitem{coco} G. M. Coclite and M. M. Coclite, On the summability of weak solutions for a singular Dirichlet problem in bounded domains, Adv. Differential Equations 19 (5-6) (2014) 585-612.

\bibitem{copal} M. M. Coclite and G. Palmieri, On a singular nonlinear Dirichlet problem, Comm. Partial Differential Equations 14 (10) (1989) 1315-1327. 

\bibitem{ch} S. Cong and Y. Han, Compatibility conditions for the existence of weak solutions to a singular elliptic equation, Bound. Value Probl. 2015, 2015:27, 11 pp.
	\bibitem{crt} M. G. Crandall, P. H. Rabinowitz and L. Tartar, On a dirichlet problem with a singular nonlinearity, Comm.  Part. Diff. Eq. 2 (2) (1977) 193-222.

\bibitem{do} L. M. De Cave, F. Oliva, Elliptic equations with general singular lower order term and measure data, Nonlinear Analysis 128 (2016) 391-411.
	\bibitem{ddo}	L. M. De Cave, R. Durastanti and F. Oliva, Existence and uniqueness results for possibly singular nonlinear elliptic equations with measure data, Nonlinear Differ. Equ. Appl. (2018) 25:18.
	
\bibitem{dur} R. Durastanti, Asymptotic behavior and existence of solutions for singular
elliptic equations, Annali di Matematica Pura ed Applicata (1923 -), in press
, DOI:10.1007/s10231-019-00906-0.
	
	
\bibitem{diazjfa} J.I. Diaz and J.M. Rakotoson, On the differentiability of very weak solutions with right-hand side data integrable with respect to the distance to the boundary, J. Funct. Anal. 257 (2009) 807-831.
\bibitem{fs} F. Faraci and G. Smyrlis, Three solutions for a singular quasilinear elliptic problem, Proc. Edinb. Math. Soc. (2) 62 (1) (2019) 179-196. 
\bibitem{gmm} D. Giachetti, P.J. Mart\'inez-Aparicio and F. Murat, A semilinear elliptic equation with a mild singularity
at u = 0: Existence and homogenization, J. Math. Pures Appl.  107 (2017) 41-77.
\bibitem{GMM} D. Giachetti, P. J. Mart\'inez-Aparicio and F. Murat, Definition, existence, stability and uniqueness of the solution to a semilinear elliptic problem with a strong singularity at $ u = 0 $, Ann. Scuola Normale Pisa (5) 18 (4) (2018) 1395-1442.
\bibitem{gps} D. Giachetti, F. Petitta and S. Segura de Le\'on, A priori estimates for elliptic problems with a strongly singular gradient term and a general datum, Differential  and Integral Equations, 26 (9/10) (2013) 913-948
\bibitem{gg} T. Godoy and A. Guerin, Existence of nonnegative solutions to singular elliptic problems, a variational approach, Discrete Contin. Dyn. Syst. 38 (3) (2018) 1505-1525.


\bibitem{ghl} C. Gui and F. Lin, Regularity of an elliptic problem with a singular nonlinearity, Proc. Roy. Soc. Edinburgh Sect. A 123 (6) (1993)  1021-1029.

\bibitem{ll} J. Leray and J. L. Lions, Quelques r\'esulatats de Vi\text{$\check{s}$}ik sur les probl\'emes elliptiques nonlin\'eaires par les m\'ethodes de Minty-Browder, Bull. Soc. Math. France 93 (1965) 97-107.
	\bibitem{lm} A. C. Lazer and P. J. McKenna,  On a singular nonlinear elliptic boundary-value problem, Proc. Amer. Math. Soc. (1991)  111 (3) 721-730.
\bibitem{locsc}	N. H. Loc and K. Schmitt, Boundary value problems for singular elliptic equations, Rocky Mountain J. Math. 41 (2) (2011).
\bibitem{o}	F. Oliva, Regularizing effect of absorption terms in singular problems, Journal of mathematical analysis and applications, 472 (1) (2019) 1136-1166. 	 
	\bibitem{op} F. Oliva and F. Petitta, On singular elliptic equations with measure sources, ESAIM Control Optim. Calc. Var. 22 (2016) 289-308.
	\bibitem{op2} F. Oliva and F. Petitta, Finite and infinite energy solutions of singular elliptic problems: 	existence and uniqueness, Journal of Differential Equations 264 (1) (2018) 311-340.	
	\bibitem{sak} S. Sakaguchi, Concavity properties of solutions to some degenerate quasilinear elliptic Dirichlet problems, Ann. Scuola Norm. Sup. Pisa Cl. Sci. 14 (1987) 403-421.
	\bibitem{santos} C. A. Santos and L. Santos, How to break the uniqueness of $W^{1,p}_{\rm loc}(\Omega)$-solutions for very singular elliptic problems by non-local terms, Z. Angew. Math. Phys. 69 (6) (2018) Art. 145, 22 pp.
\bibitem{sy} J. Shi and M. Yao, On a singular nonlinear semilinear elliptic problem, Proc. Roy. Soc. Edinburgh Sect. A 128 (6) (1998) 1389-1401.

\bibitem{sta} G. Stampacchia, Le probl\`eme de Dirichlet pour les \'equations elliptiques du seconde ordre \`a coefficients discontinus, Ann. Inst. Fourier (Grenoble) 15 (1965) 189-258.

\bibitem{stuart} C.A. Stuart, Existence and approximation of solutions of nonlinear elliptic problems, Mathematics Report, Battelle Advanced Studies Center, Geneva, Switzerland, (1976) 86.

\bibitem{s} Y. Sun, Compatibility phenomena in singular problems, Proc. Roy. Soc. Edinburgh Sect. A 143 (6) (2013) 1321-1330.

\end{thebibliography}
\end{document}